\documentclass[a4paper,12pt, twoside, reqno]{amsart}
\usepackage{amsmath}%
\usepackage{amsfonts}%
\usepackage{amssymb}%
\usepackage{graphicx}
\newtheorem{theorem}{Theorem}
\theoremstyle{plain}

\newtheorem{definition}{Definition}
\newtheorem{example}{Example}

\newtheorem{lemma}{Lemma}

\newtheorem{remark}{Remark}

\newtheorem{ass}{Assumption}[section]
\numberwithin{equation}{section}
\newcommand{\R}{\mathbb R}

\begin{document}
\title{Generalized coupled fixed point theorems for mixed monotone mappings in partially ordered metric spaces}
\author{Vasile Berinde}

\begin{abstract}
In this paper we extend the coupled fixed point theorems for mixed monotone operators $F:X \times X \rightarrow  X$ obtained in  [T.G. Bhaskar, V. Lakshmikantham, \textit{Fixed point theorems in partially ordered metric spaces and applications}, Nonlinear Anal. TMA \textbf{65} (2006) 1379-1393] by significantly weakening the involved contractive condition. Our technique of proof is essentially different and more natural. An example as well an application to periodic BVP are also given in order to illustrate the effectiveness of our generalizations. 
\end{abstract}
\maketitle

\pagestyle{myheadings} \markboth{Vasile Berinde} {Generalized coupled fixed point theorems for mixed monotone mappings}

\section{Introduction} 
The Banach contraction principle, which is the most famous metrical fixed point theorem, play a very important role in nonlinear analysis. It basically asserts that, if $(X,d)$ is a complete metric space and $T:X\rightarrow X$ is a contraction, i.e., there exists a constant $c \in [0,1)$ such that
\begin{equation} \label{Banach}
d(Tx, Ty)\leq c\,d(x,y),\,\text{for all}\;x,y\in X\,,
\end{equation}
then $T$ has a unique fixed point $x^\ast$ in $X$, i.e., $T(x^\ast)=x^\ast$ and the Picard iteration $\{x_n\}^\infty_{n=0}$
defined by
\begin{equation} \label{Picard}
x_{n+1}=Tx_n\,,\quad n=0,1,2,\dots
\end{equation}
converges to $x^\ast$, for any $x_0\in X$.

The Banach contraction principle has been generalized in several directions, see for example \cite{Ber07} and \cite{Rus2}. Another recent  direction of such generalizations, see \cite{Aga}, \cite{LakC}-\cite{Ran}, has been obtained by weakening the requirements in the contractive condition \eqref{Banach} and, in compensation, by simultaneously enriching the metric space structure  with a partial order. 
The first result of this kind is the following fixed point theorem for monotone mappings in ordered metric spaces obtained in \cite{Ran}.  

\begin{theorem}[Ran and Reurings \cite{Ran}]\label{th1}

Let $X$  be a partially ordered set such that every pair $x, y \in X$  has a lower  bound and an upper bound. Furthermore,  let $d$ be a metric on $X$  such that $(X, d)$ is a complete metric space. If $T$ is a continuous, monotone (i.e., either order-preserving or order-reversing) map from $X$  into $X$  such that
\begin{equation} \label{Cond-Ran}
\exists 0<c<1:\,d(T(x), T(y))\leq c\,d(x,y),\,\forall x\geq y,
\end{equation}  
\begin{equation} \label{exist}
\exists x_0\in X:\,x_0\leq T(x_0) \textnormal{ or } x_0\geq T(x_0),
\end{equation}
then $T$ has a unique fixed point $\overline{x}$. Moreover, for every $x \in X$,
$$
\lim_{n \rightarrow \infty} T^n(x)=\overline{x}.
$$
	\end{theorem}

 \begin{remark} \em
Note first that, while condition \eqref{Banach} must be fulfilled for all $x,y\in X$, condition  \eqref{Cond-Ran} requires only to be  satisfied for all $x,y\in X$ with $ x\geq y$. Secondly, Theorem \ref{th1} was shown to have important applications in solving nonlinear matrix equations \cite{Ran}. Note that in the absence of the assumption "every pair $x, y \in X$  has a lower  bound and an upper bound", the fixed point in Theorem \ref{th1} may not be unique.
 \end{remark}

Following basically the same approach as the one in \cite{Ran}, and unifying the results in \cite{Nie06} and \cite{Nie07}, Bhaskar and Lakshmikantham \cite{Bha} obtained some coupled fixed point results for mixed monotone operators  $F:X \times X \rightarrow  X$ which satisfy a certain contractive type condition, where $X$ is a partially ordered metric space. 





In order to state the main result in \cite{Bha}, we need the following notions. Let$\left(X,\leq\right)$ be a partially ordered set and  endow the product space $X \times X$ with the following partial order:$$\textnormal{for } \left(x,y\right), \left(u,v\right) \in X \times X,  \left(u,v\right) \leq \left(x,y\right) \Leftrightarrow x\geq u,  y\leq v.$$

We say that a mapping $F:X \times X \rightarrow X$ has the \textit{mixed monotone property} if $F\left(x,y\right)$ is monotone nondecreasing in $ x$ and is monotone non increasing in $y$, that is, for any $x, y \in X,$   
$$  x_{1}, x_{2} \in X,  x_{1} \leq  x_{2} \Rightarrow  F\left(x_{1},y\right) \leq  F\left(x_{2},y\right) $$
  and, respectively,
$$ y_{1},  y_{2} \in  X, y_{1} \leq y_{2} \Rightarrow F\left(x,y_{1}\right)  \geq F\left(x,y_{2}\right). $$

A pair $ \left(x,y\right) \in X \times X $ is called a \textit{coupled  fixed  point} of the mapping $F$ if                                              $$ F\left(x,y\right) = x,  F\left(y,x\right) = y.$$
The next theorem is the main existence result in \cite{Bha}.

\begin{theorem}[Bhaskar and Lakshmikantham \cite{Bha}]\label{th2}
	Let  $\left(X,\leq\right)$ be a partially ordered set and suppose there is a metric $d$  on $X$ such that $\left(X,d\right)$  is a complete metric space. Let $F : X \times X \rightarrow X $ be a continuous mapping having the mixed monotone property on $X$. Assume that there exists a constant $k \in \left[0,1\right)$  with                                                                                   
\begin{equation} \label{Bhas}
	d\left(F\left(x,y\right),F\left(u,v\right)\right) \leq \frac{k}{2}\left[d\left(x,u\right) + d\left(y,v\right)\right],\textnormal{ for each }x \geq u, y \leq  v.
\end{equation}  
If there exist $x_{0}, y_{0} \in X$ such that  
$$                                                                                       
x_{0} \leq F\left(x_{0},y_{0}\right)\textrm{ and }y_{0} \geq F\left(y_{0},x_{0}\right),
$$
 then there exist $x, y \in X$ such that $$x = F\left(x,y\right)\textnormal{ and }y = F\left(y,x\right).$$ 
	\end{theorem}        
	
In \cite{Bha} Bhaskar and Lakshmikantham also established some uniqueness results for coupled fixed points, as well as existence of fixed points of $F$ ($x$ is a fixed point of $F$ if $F(x,x)=x$).

Starting from the results in \cite{Bha}, our main aim in this paper is to obtain more general coupled fixed point theorems for mixed monotone operators $F : X \times X \rightarrow X $  satisfying a contractive condition which is significantly more general than \eqref{Bhas}. 

Our approach brings at least five new features to the coupled fixed point theory: first, we weaken the contractive condition satisfied by $F$; secondly, we do not assume the continuity of $F$; third, our technique of proof is simpler and essentially different from the one used in \cite{Bha} and in the numerous papers devoted to coupled fixed point problems that appeared in the last years; fourth, we provide a method for approximating the coupled fixed points and, fifth, we also provide an error estimate for this method.                                                                        
	
\section{Main results}
	
The first main result in this paper is the following coupled fixed point result which generalize Theorem \ref{th2} (Theorem 2.1 in \cite{Bha}).

\begin{theorem}\label{th3}
	Let  $\left(X,\leq\right)$ be a partially ordered set and suppose there is a metric $d$  on $X$ such that $\left(X,d\right)$  is a complete metric space. Let $F : X \times X \rightarrow X $ be a  mixed monotone mapping for which there exists a constant $k \in \left[0,1\right)$ such that for each $x \geq u, y \leq  v$,                                                                       
\begin{equation} \label{Bhas1}
d\left(F\left(x,y\right),F\left(u,v\right)\right)+d\left(F\left(y,x\right),F\left(v,u\right)\right)  \leq k \left[d\left(x,u\right) + d\left(y,v\right)\right],
\end{equation}  
If there exist $x_{0}, y_{0} \in X$ such that  
\begin{equation} \label{mic}
 x_{0} \leq F\left(x_{0},y_{0}\right)\textrm{ and }y_{0} \leq F\left(y_{0},x_{0}\right),
\end{equation} or
\begin{equation} \label{mare} 
x_{0} \geq F\left(x_{0},y_{0}\right)\textrm{ and }y_{0} \leq F\left(y_{0},x_{0}\right),
\end{equation}
 then there exist $\overline{x}, \overline{y} \in X$ such that $$\overline{x} = F\left(\overline{x},\overline{y}\right)\textrm{and }\overline{y} = F\left(\overline{y},\overline{x}\right).$$ 
	\end{theorem}

\begin{proof}
Consider the functional $d_2:X^2\times X^2 \rightarrow \mathbb{R}_{+}$ defined by
$$
d_2(Y,V)=\frac{1}{2}\left[d(x,u)+d(y,v)\right],\,\forall Y=(x,y),V=(u,v) \in X^2.
$$
It is a simple task to check that $d_2$ is a metric on $X^2$ and, moreover, that, if $(X,d)$ is complete, then $(X^2,d_2)$ is a complete metric space, too.
Now consider the operator $T:X^2\rightarrow X^2$ defined by
$$
T(Y)=\left(F(x,y),F(y,x)\right),\,\forall Y=(x,y) \in X^2.
$$
Clearly, for $Y=(x,y),\,V=(u,v)\in X^2$, in view of the definition of $d_2$, we have
$$
d_2(T(Y),T(V))=\frac{d\left(F\left(x,y\right),F\left(u,v\right)\right)+d\left(F\left(y,x\right),F\left(v,u\right)\right)}{2}
$$
and
$$
d_2(Y,V)=\frac{d\left(x,u\right) + d\left(y,v\right)}{2}.
$$
Hence, by the contractive condition \eqref{Bhas1} we obtain a Banach type contraction condition: 
\begin{equation} \label{contr}
d_2(T(Y), T(V))\leq k\,d_2(Y,V),\,\forall Y,V \in X^2,\textnormal{ with }Y\geq V.
\end{equation} 
Assume \eqref{mic} holds (the case \eqref{mare} is similar). Then, there exists $x_0,y_0\in X$ such that
$$
x_0\leq F(x_0,y_0) \textnormal{ and } y_0\geq F(y_0,x_0). 
$$
Denote $Z_0=(x_0,y_0)\in X^2$ and consider the Picard iteration associated to $T$ and to the initial approximation $Z_0$, that is, the sequence $\{Z_n\}\subset X^2$ defined by
\begin{equation} \label{eq-3}
Z_{n+1}=T (Z_n),\,n\geq 0,
\end{equation}
where $Z_n=(x_n,y_n)\in X^2,\,n\geq 0$.

Since $F$ is mixed monotone, we have
$$
Z_0=(x_0,y_0)\leq (F(x_0,y_0), F(y_0,x_0))=(x_1,y_1)=Z_1
$$
and, by induction,
$$
Z_n=(x_n,y_n)\leq (F(x_n,y_n), F(y_n,x_n))=(x_{n+1},y_{n+1})=Z_{n+1},
$$
which shows that $T$ is monotone and the sequence $\{Z_n\}_{n=0}^{\infty}$ is non decreasing.
We follow now the steps in the proof of Banach's contraction fixed point theorem. Take $Y=Z_n\geq Z_{n-1}=V$ in \eqref{contr} and obtain
$$
d_2(T(Z_{n}),T(Z_{n-1})\leq k \cdot d_2 (Z_n, Z_{n-1}),\,n\geq 1,
$$
that is,
\begin{equation} \label{eq-4}
d_2(Z_{n+1},Z_n)\leq k \cdot d_2 (Z_n, Z_{n-1}),\,n\geq 1,
\end{equation}
which, by induction, gives
\begin{equation} \label{eq-5}
d_2(Z_{n+1},Z_n)\leq k^n \cdot d_2 (Z_1, Z_{0}),\,n\geq 1.
\end{equation}
We claim that $\{Z_n\}_{n=0}^{\infty}$ is a Cauchy sequence in  $(X^2,d_2)$. Let $n<m$. Then by \eqref{eq-5}
$$
d_2(Z_{n},Z_m)\leq \sum_{i=n+1}^{m} d_2 (Z_i, Z_{i-1})\leq \left(k^n+k^{n+1}+\dots+k^{m-n-1}\right)d_2 (Z_1, Z_{0})
$$
$$
\leq k^n \frac{1-k^{m-n-1}}{1-k}d_2 (Z_1, Z_{0}).
$$
So, $\{Z_n\}_{n=0}^{\infty}$ is indeed a Cauchy sequence in the complete metric space $(X^2,d_2)$ and hence, convergent: there exists $\overline{Z}\in X^2$ such that
$$
\lim_{n\rightarrow \infty} Z_n=\overline{Z}.
$$ 
Because $T$ is continuous in $(X^2,d_2)$ , by virtue of the Lipschitzian type condition \eqref{Bhas1}, by \eqref{eq-3} it follows that $\overline{Z}$ is a fixed point of $T$, that is,
$$
T(\overline{Z})=\overline{Z}.
$$
Let $\overline{Z}=(\overline{x},\overline{y})$. Then, by the definition of $T$, this means
$$\overline{x} = F\left(\overline{x},\overline{y}\right)\textrm{and }\overline{y} = F\left(\overline{y},\overline{x}\right),$$ 
that is,  $(\overline{x},\overline{y})$ is a coupled fixed point of $F$.
\end{proof}

 \begin{remark} \em
Theorem \ref{th3} is more general than Theorem \ref{th2}, since the contractive condition \eqref{Bhas1} is weaker than \eqref{Bhas}, a fact which is clearly illustrated by the next example.
 \end{remark}
 
\begin{example} \label{ex1}
Let $X=\mathbb{R},$ $d\left(x,y\right)=|x-y|$ and $F:X\times X \rightarrow X$ be defined by 
$$
F\left(x,y\right)=\frac{x-3y}{5}, \,(x,y)\in X^2.
$$ 
Then $F$ is mixed monotone and satisfies condition \eqref{Bhas1}  but  does not satisfy condition \eqref{Bhas}. 
Indeed, assume there exists $k$, $0\leq k <1$, such that  \eqref{Bhas} holds. This means
$$
\left|\frac{x-3y}{5}-\frac{u-3v}{5}\right| \leq \frac{k}{2}\left[\left|x-u\right|+\left|y-v\right|\right],\,x\geq u,\,y\leq v,
$$
by which, for $x=u$, we get
$$
\frac{3}{5}\left|y-v\right| \leq \frac{k}{2}\left|y-v\right|,\,y\leq v,
$$
which for $y<v$ would imply $\dfrac{3}{5}\leq \dfrac{k}{2}\Leftrightarrow \dfrac{6}{5}\leq k <1$, a contradiction.

Now we prove that \eqref{Bhas1} holds. Indeed, since we have
$$
\left|\frac{x-3y}{5}-\frac{u-3v}{5}\right| \leq \frac{1}{5}\left|x-u\right|+\frac{3}{5}\left|y-v\right|,\,x\geq u,\,y\leq v,
$$
and
$$
\left|\frac{y-3x}{5}-\frac{v-3u}{5}\right| \leq \frac{1}{5}\left|y-v\right|+\frac{3}{5}\left|x-u\right|,\,x\geq u,\,y\leq v,
$$
by summing up the two inequalities above we get exactly  \eqref{Bhas1} with $k=\dfrac{4}{5}<1$. Note also that $x_0=-3,\,y_0=3$ satisfy \eqref{mic}.

So  by Theorem \ref{th3} we obtain that $F$ has a (unique) coupled fixed point $(0,0)$ but Theorem \ref{th2} cannot be applied to $F$ in this example.
\end{example} 

\begin{remark} \em
As suggested by Example \ref{ex1}, let us note that, since the contractivity condition \eqref{Bhas1} is valid  only for comparable elements in  $X^2$, Theorem \ref{th3} cannot guarantee in general the uniqueness of the coupled fixed point. 
\end{remark}

It is therefore our interest now to find additional conditions to ensure that the coupled fixed point in Theorem \ref{th3} is in fact unique. Such a condition is the one used in Theorem 2.2 of Bhaskar and Lakshmikantham \cite{Bha} or in Theorem \ref{th2} of Ran and Reurings \cite{Ran}:

\textnormal{every pair of elements in} $X^2$ \textnormal{has either a lower bound or an upper bound},
which is known, see \cite{Bha}, to be equivalent to the following condition: for all $Y=(x,y),\,\overline{Y}=(\overline{x},\overline{y})\in X^2$,
\begin{equation} \label{eq-7}
\exists Z=(z_1,z_2)\in X^2 \textnormal{ that is comparable to }  Y \textnormal{ and } \overline{Y}.
\end{equation}

\begin{theorem} \label{th4}
Adding  condition  \eqref{eq-7}  to  the  hypotheses  of   Theorem \ref{th3},  we  obtain  the
uniqueness of the coupled fixed point of  $F$.
\end{theorem}

\begin{proof} In search for a contradiction, assume that $Z^*=(x^*,y^*)\in X^2$ is a coupled fixed point of $F$, different from $\overline{Z}=(\overline{x},\overline{y})$. This means that $d_2(Z^*,\overline{Z})>0.$ We discuss two cases:

Case 1. $Z^*$ is comparable to $\overline{Z}$.

As $Z^*$ is comparable to $\overline{Z}$ with respect to the ordering in $X^2$, by taking in \eqref{contr} $Y=Z^*$ and $V=\overline{Z}$ (or $V=Z^*$ and $Y=\overline{Z}$), we obtain
$$
d_2(T(Z^*),T(\overline{Z}))=d_2(Z^*,\overline{Z})\leq k\cdot d_2(Z^*,\overline{Z}),
$$
which, since $0\leq k<1$, leads to the contradiction $d_2(Z^*,\overline{Z})\leq 0$.

Case 2. $Z^*$ and $\overline{Z}$ are not comparable.

In this case, there exists an upper bound or a lower bound $Z=(z_1,z_2)\in X^2$ of $Z^*$ and $\overline{Z}$. Then, in view of the monotonicity of $T$, $T^n(Z)$ is comparable to $T^n(Z^*)=Z^*$ and to $T^n(\overline{Z})=\overline{Z}$.

Now, again by the contraction condition \eqref{contr}, we have
$$
d_2(Z^*,\overline{Z})=d_2(T^n(Z^*),T^n(\overline{Z}))\leq d_2(T^n(Z^*),T^n(Z))+d_2(T^n(Z),T^n(\overline{Z}))
$$
$$
\leq k^n\left[d_2(Z^*,Z)+d_2(Z,\overline{Z})\right]\rightarrow 0
$$
as $n\rightarrow \infty$,

which leads to the contradiction $0<d_2(Z^*,\overline{Z})\leq 0$.
\end{proof}

Similarly to \cite{Bha}, by assuming the same condition as in Theorem \ref{th4} but with respect to the ordered set $X$, that is, by assuming that every pair of elements of $X$ have either an upper bound or a lower bound in $X$, one can show that even the components of the coupled fixed points are equal. 
 
\begin{theorem} \label{th5}
In addition to the hypothesis of  Theorem \ref{th3}, suppose that every pair of elements of  $X$ has an upper bound or a lower bound in $X$. Then for the coupled fixed point $(\overline{x},\overline{y})$ we have $\overline{x} = \overline{y}$, that is, $F$ has a fixed point:
$$
F(\overline{x},\overline{x})=\overline{x}.
$$  
\end{theorem}

\begin{proof}
Let $(\overline{x},\overline{y})$ be a coupled fixed point of $F$. We consider again two cases.

Case 1. If $\overline{x},\overline{y}$ are comparable, then $
F(\overline{x},\overline{y})=\overline{x}$  is comparable to $\overline{y}=F(\overline{y},\overline{x})$ and hence, by taking $x:=\overline{x},\,y:=\overline{y},\,u:=\overline{y},\,v:=\overline{x},\,$ in \eqref{Bhas1} one obtains
\begin{equation} \label{eq-9}
d(F(\overline{x},\overline{y}),F(\overline{y},\overline{x}))\leq k\cdot d(\overline{x},\overline{y}),
\end{equation}
and so
$$
d(\overline{x},\overline{y})=d(F(\overline{x},\overline{y}),F(\overline{y},\overline{x}))\leq k\cdot d(\overline{x},\overline{y}),
$$
which by $0\leq k <1$, yields $d(\overline{x},\overline{y})=0.$

Case 2. If $\overline{x},\overline{y}$ are not comparable, then there exists a $z\in X$ comparable to $\overline{x}$ and $\overline{y}$. Suppose $\overline{x}\leq z$ and $\overline{y}\leq z$ (the other case is similar). Then in view of the order on $X^2$, it follows that 
$$
(\overline{x},\overline{y})\geq (\overline{x},z);\,(\overline{x},z)\leq (z,\overline{x});\,(z,\overline{x})\geq (\overline{y},\overline{x}),
$$
that is $(\overline{x},\overline{y}),\, (\overline{x},z);\,(\overline{x},z),\,(z,\overline{x});\,(z,\overline{x}),\, (\overline{y},\overline{x})$ are comparable in $X^2$. Now, remind from the proof of Theorem \ref{th3} that by virtue of \eqref{eq-3} and \eqref{eq-5}, for any two comparable elements $Y,V$ in $X^2$, one has
\begin{equation} \label{eq-8}
d_2(T^n(Y),T^n(V))\leq k^n d_2(Y,V),
\end{equation}
where $T$ was defined in the proof of Theorem \ref{th3}.

Now use \eqref{eq-8} for the comparable pairs $Y=(\overline{x},\overline{y}),\, V=(\overline{x},z);\,Y=(\overline{x},z),\,V=(z,\overline{x});\,Y=(z,\overline{x}),\, V=(\overline{y},\overline{x})$, respectively, to get

\begin{equation} \label{eq-10}
d_2(T^n(\overline{x},\overline{y}),T^n(\overline{x},z))\leq k^n d_2((\overline{x},\overline{y}),(\overline{x},z),
\end{equation}

\begin{equation} \label{eq-11}
d_2(T^n(\overline{x},z),T^n(z,\overline{x})\leq k^n d_2((\overline{x},z),(z,\overline{x})),
\end{equation}

\begin{equation} \label{eq-12}
d_2(T^n(z,\overline{x}),T^n(\overline{y},\overline{x}))\leq k^n d_2((z,\overline{x}),(\overline{y},\overline{x})).
\end{equation}
Now, by using the triangle inequality and \eqref{eq-10}, \eqref{eq-11}, \eqref{eq-12}, one has
$$
d(\overline{x},\overline{y})=\frac{d(\overline{x},\overline{y})+d(\overline{x},\overline{y})}{2}=d_2((\overline{x},\overline{y}),(\overline{y},\overline{x}) )=d_2(T^n(\overline{x},\overline{y}),T^n(\overline{y},\overline{x}))
$$
$$
\leq d_2(T^n(\overline{x},\overline{y}),T^n(\overline{x},z))+d_2(T^n(\overline{x},z),T^n(z,\overline{x}))+d_2(T^n(z,\overline{x}),T^n(\overline{y},\overline{x}))
$$
$$
\leq k^n \left[d_2((\overline{x},\overline{y}),(\overline{x},z)+d_2((\overline{x},z),(z,\overline{x}))+d_2((z,\overline{x}),(\overline{y},\overline{x}))\right]=
$$
$$
\leq k^n \left[d(\overline{y},z)+d(\overline{x},z)\right]\rightarrow 0,\,\textnormal{ as } n\rightarrow\infty,
$$
which shows that $d(\overline{x},\overline{y})=0$, that is $\overline{x}=\overline{y}$.
\end{proof}

Similarly, one can obtain the same conclusion under the following alternative assumption.

\begin{theorem} \label{th6}
In addition to the hypothesis of  Theorem \ref{th3}, suppose that $x_0,y_0 \in X$ are comparable. Then for the coupled fixed point $(\overline{x},\overline{y})$ we have $\overline{x} = \overline{y}$, that is, $F$ has a fixed point:
$$
F(\overline{x},\overline{x})=\overline{x}.
$$   
\end{theorem}

\begin{proof}
Assume we are in the case \eqref{mic}, that is 
$$
 x_{0} \leq F\left(x_{0},y_{0}\right)\textrm{ and }y_{0} \leq F\left(y_{0},x_{0}\right).
$$
Since $x_0,y_0$ are comparable, we have $x_0\leq y_0$ or $x_0\geq y_0$. Suppose we are in the first case. Then, by the mixed monotone property of $F$, we have
$$
 x_{1} =F\left(x_{0},y_{0}\right) \leq F\left(y_{0},x_{0}\right)=y_1,
$$
and, hence, by induction one obtains
\begin{equation} \label{eq-13}
x_n\leq y_n,\,n\geq 0.
\end{equation}
Now, since
$$
\overline{x}=\lim_{n\rightarrow\infty} F(x_n,y_n);\,\overline{y}=\lim_{n\rightarrow\infty} F(y_n,x_n),
$$
by the continuity of the distance $d$, one has
$$
d(\overline{x},\overline{y})=d(\lim_{n\rightarrow\infty} F(x_n,y_n),\lim_{n\rightarrow\infty} F(y_n,x_n))=\lim_{n\rightarrow\infty} d(F(x_n,y_n), F(y_n,x_n))
$$
$$
=\lim_{n\rightarrow\infty} d(x_{n+1},y_{n+1}).
$$
On the other hand, by taking $Y=(x_n,y_n),\,V=(y_n,x_n)$ in \eqref{Bhas1} we have
$$
d(F(x_n,y_n),F(y_n,x_n))\leq k d(x_n,y_n),\,n\geq 0,
$$
which actually means
$$
d(x_{n+1},y_{n+1})\leq k d(x_n,y_n),\,n\geq 0.
$$
Therefore
$$
d(\overline{x},\overline{y})=\lim_{n\rightarrow\infty} d(x_{n+1},y_{n+1})\leq \lim_{n\rightarrow\infty} k^n d(x_1,y_1)=0.
$$

\end{proof}

\begin{remark} \em
We note that for all Theorems \ref{th3}-\ref{th6} we actually can approximate the coupled fixed point $(\overline{x},\overline{y})$. 

Indeed, let us denote $Z_0=(x_0,y_0)\in X^2$ and let $\{Z_n\}\subset X^2$ be the Picard iteration  associated to $T$ and to the initial approximation $Z_0$, that is, the sequence $\{Z_n\}\subset X^2$ defined by
\begin{equation} \label{eq-30}
Z_{n+1}=T (Z_n),\,n\geq 0,
\end{equation}
where $Z_n=(x_n,y_n)\in X^2,\,n\geq 0$.

Then, $x_{n+1}=F(x_n,y_n)$, $y_{n+1}=F(y_n,x_n)$ and, by the proof of Theorem \ref{th3} we have
$$
\overline{x}=\lim_{n\rightarrow \infty} x_n; \,\overline{y}=\lim_{n\rightarrow \infty} y_n,
$$
that is,
$$
\overline{x}=\lim_{n\rightarrow\infty} F(x_n,y_n);\,\overline{y}=\lim_{n\rightarrow\infty} F(y_n,x_n),
$$
and the following error estimate holds
$$
d(x_n,\overline{x})+d(y_n,\overline{y})\leq \frac{k^n}{1-k}\left[d(x_1,x_0)+d(y_1,y_0)\right],\,n\geq 0.
$$
This can be also written as
$$
d_2\left((x_n,y_n),(\overline{x},\overline{y})\right)\leq \frac{k^n}{1-k}d_2\left((x_1,y_1), (x_0,y_0)\right),\,n\geq 0.
$$
One can similarly obtain an \textit{a posteriori} error estimate for the Picard iteration \eqref{eq-30}, see \cite{Ber07}.
\end{remark}

\section{Applications to periodic boundary value problems}

In order to compare our existence and uniqueness results established in Section 2 to the ones of \cite{Bha}, we shall consider the same periodic boundary value problem studied there, that is,
\begin{equation}\label{3.1}
    u'=h(t,u),\quad t\in I=(0,T)
\end{equation}
\begin{equation}\label{3.2}
    u(0)=u(T).
\end{equation}
Like in \cite{Bha}, we assume that there exist the continuous functions $f, g$ such that
$$  h(t,u)=f(t,u)+g(t,u),\quad t\in[0,T],  $$
where $f$ and $g$ fulfill the following conditions:
\begin{ass} \label{as1}
 There exist the positive numbers $\lambda_1, \lambda_2, \mu_1, \mu_2$, such that for all $u,v\in\R$, $v\leq u$,
\begin{equation}\label{3.3}
   0\leq (f(t,u)+\lambda_1 u)-(f(t,v)+\lambda_1 v)\leq \mu_1(u-v) 
\end{equation}
\begin{equation}\label{3.4}
    -\mu_2(u-v)\leq(g(t,u)-\lambda_2 u)-(g(t,v)-\lambda_2 v)\leq 0,
\end{equation}
where
\begin{equation}\label{3.4-1}
\frac{\mu_1+\mu_2}{\lambda_1+\lambda_2}<1.
\end{equation}
\end{ass}
In order to obtain the (unique) solution of Eqs. (\ref{3.1}) and (\ref{3.2}), we first study the existence of a solution of the following periodic system
\begin{equation}\label{3.5}
    u'+\lambda_1 u-\lambda_2 v=f(t,u)+g(t,v)+\lambda_1 u-\lambda_2 v
\end{equation}
\begin{equation}\label{3.6}
    v'+\lambda_1 v-\lambda_2 u=f(t,v)+g(t,u)+\lambda_1 v-\lambda_2 u
\end{equation}
together with the boundary conditions
\begin{equation}\label{3.7}
    u(0)=u(T)\;\;\text{and}\;\; v(0)=v(T).
\end{equation}
As shown in \cite{Bha}, the problem (\ref{3.5}-(\ref{3.7}) is equivalent to the system of 
integral equations
\begin{align*}
    u(t) & =\int^T_0 G_1(t,s)[f(s,u)+g(s,v)+\lambda_1 u-\lambda_2 v]+\\
    &\quad + G_2(t,s)[f(s,v)+g(s,u)+\lambda_1 v-\lambda_2 u]ds,
\end{align*}
\begin{align*}
    v(t) & =\int^T_0 G_1 (t,s)[f(s,v)+g(s,u)+\lambda_1 v-\lambda_2 u]+\\
    &\quad + G_2 (t,s)[f(s,u)+g(s,v)+\lambda_1 u-\lambda_2 v]ds,
\end{align*}
where, for $i=1,2$, 
$$  G_i(t,s)=\frac{1}{2}\left(\frac{e^{\tau_i(t-s)}}{1-e^{\tau_1 T}}+
    \frac{e^{\tau_{i+1}(t-s)}}{1-e^{\tau_2 T}}\right),\;\;\text{if}\;\;
    0\leq s<t\leq T
$$
and
$$  G_i(t,s)=\frac{1}{2}\left(\frac{e^{\tau_i(t+T-s)}}{1-e^{\tau_1 T}}+
    \frac{e^{\tau_{i+1}(t+T-s)}}{1-e^{\tau_2 T}}\right),\;\;
    \text{if}\;\; 0\leq t<s\leq T,
$$
with  $\tau_1=-(\lambda_1+\lambda_2)$, $\tau_2=\lambda_2-\lambda_1$, and $\tau_3=\tau_1$. 

We shall need Lemma 3.2 from \cite{Bha}.
\begin{lemma} \label{lem1} If 
\begin{equation}\label{3.8}
    \ln\frac{2e-1}{e}\leq (\lambda_2-\lambda_1)T
\end{equation}
and
\begin{equation}\label{3.9}
    (\lambda_1+\lambda_2)T\leq 1
\end{equation}
then $G_1(t,s)\geq 0$ and $G_2(t,s)\leq 0$, for $0\leq t$, $s\leq T$.
\end{lemma}

Consider now $X=C(I, \R)$ be the metric space of all continuous functions $u:I\to\R$, endowed with the sup metric:
$$  d(u,v)=\sup_{t\in I}|u(t)-v(t)|,\;\; \text{for}\;\; u,v\in X.  $$
Then the corresponding metric $d_2$ on $X^2$ is defined by
$$  d_2((u_1, v_1), (u_2, v_2))=\frac{1}{2}\left[\sup_{t\in I}
    |u_1(t)-u_2(t)|+\sup_{t\in I}|v_1(t)-v_2(t)|\right].
$$
Also consider on $X^2$ the partial order relation:
$$  (u_1, v_1)\leq (u_2, v_2)\Leftrightarrow u_1(t)\leq u_2(t)\;\;
    \text{and}\;\; v_1(t)\geq v_2(t),\; t\in I,
$$
and define for $t\in I$ the operator
\begin{align}\label{3.10}
    A[u,v](t) & =\int^T_0 G_1(t,s)[f(s,u)+g(s,v)+\lambda_1 u-\lambda_2 v]\nonumber\\
    &\quad + G_2(t,s)[f(s,v)+g(s,u)+\lambda_1 v-\lambda_2 u]ds.
\end{align}
It is obvious that, if $(u,v)\in X^2$ is a coupled fixed point of $A$, then we have
$$  u(t)=A[u,v](t)\;\;\text{and}\;\; v(t)=A[v,u](t)  $$
for all $t\in I$, and $(u,v)$ is a solution of (\ref{3.5}) and (\ref{3.6}) satisfying the periodic boundary condition (\ref{3.7}).

\begin{definition}
\em (\cite{Bha}) A pair  $(\alpha, \beta)\in X^2$ is called a \emph{coupled lower-upper solution} of the periodic BVP (\ref{3.1}) and (\ref{3.2}) (and hence of (\ref{3.5})-(\ref{3.7}) if
$$  \alpha'(t)\leq f(t, \alpha(t))+g(t, \beta(t)) $$
and
$$  \beta'(t)\geq f(t, \beta(t))+g(t, \alpha(t)),\quad t\in I,$$
together with the periodic conditions
$$  \alpha(0)\leq \alpha(T)\;\;\text{and}\;\; \beta(0)\geq\beta(T).  $$
\end{definition}

The following lemma, which is Lemma 3.5 in \cite{Bha}, exhibits the relation between coupled lower-upper solutions of (\ref{3.5})-(\ref{3.7}) and the integral operator $A$ given by (\ref{3.10}).

\begin{lemma} \label{lem2} If
\begin{equation}\label{3.11}
    \lambda_1(\alpha(T)-\alpha(0))+\lambda_2(\beta(0)-\beta(T))
    \leq\frac{\alpha(T)-\alpha(0)}{T}\,,
\end{equation}
\begin{equation}\label{3.12}
    \lambda_1(\beta(0)-\beta(T))+\lambda_2(\alpha(T)-\alpha(0))\leq
    \frac{\beta(0)-\beta(T)}{T}\,,
\end{equation}
then $\alpha(t)\leq A[u(t), v(t)]$ and $\beta(t)\geq A[v(t), u(t)]$, $t\in(0,T).$
\end{lemma}
\begin{remark}\em
As noted in \cite{Bha}, the hypotheses of Lemma \ref{lem2} are satisfied, for example, if
$$  T(\lambda_1+\lambda_2)<\frac{\beta(0)-\beta(T)}{\alpha(T)-\alpha(0)}<1,
$$
and also, conditions (\ref{3.8}), (\ref{3.9}), (\ref{3.11}) and (\ref{3.12}) together with the Assumption 3.1, can all be satisfied simultaneously by making suitable choices of $\lambda_1, \lambda_2$ and $T$.
\end{remark}
 The next theorem extends Theorem 3.7 in \cite{Bha} by considering the weaker condition \eqref{3.4-1} in Assumption \ref{as1}.
\begin{theorem}
Consider the problem (\ref{3.5})-(\ref{3.7}) with $f,g\in C(I\times \R, \R)$ and suppose that the Assumption $3.1$ is satisfied. Assume there exist coupled lower-upper solutions $\alpha(t)$ and $\beta(t)$ for $(\ref{3.5})-(\ref{3.7})$ respectively, such that $(\ref{3.11})$ and $(\ref{3.12})$ hold. Further, if $(\ref{3.8})$ and $(\ref{3.9})$ are fulfilled, then there exists a unique solution of the periodic BVP $(\ref{3.5})-(\ref{3.7})$.
\end{theorem}
\begin{proof}
We first obtain the existence of a unique solution of the periodic BVB (\ref{3.5})-(\ref{3.7}) by showing that the operator $A:X\times X\to X$ has a unique coupled fixed point in $X\times X$. To this end, we verify that $A$ satisfies the hypotheses of Theorems \ref{th3} and \ref{th5}.

Like in the proof of Theorem 3.7 in \cite{Bha}, we can show that, for any $(u_1, v)\geq (u_2, v)\in X^2$,
$$  A[u_1, v](t)\geq A[u_2, v](t),   $$
and that for $(u, v_1)\leq (u, v_2)\in X^2$, one has
$$  A[u, v_1](t)\geq A[u, v_2](t),   $$
which shows that $A[u,v]$ is mixed monotone.

Now, let us consider $(x,y)$, $(u,v)\in X^2$, with $(x,y)\leq (u,v)$, that is, $u\geq x$ and $v\leq y$. 
We have
\begin{align*}
    d(A[u,v], A[x,y]) &=\sup_{t\in I}|A[u,v](t)-A[x,y](t)|\\
    &=\sup_{t\in I}\Big|\int^T_0 G_1(t,s)\big[[f(s,u)+g(s,v)+\lambda_1 u-\lambda_2 v]\\
    &\quad - [f(s,x)+g(s,y)+\lambda_1 x-\lambda_2 y]\big]\\
    &\quad + G_2(t,s)\big[[f(s,v)+g(0,u)+\lambda_1 v-\lambda_2 u]\\
    &\quad -[f(s,y)+g(s,x)+\lambda_1 y-\lambda_2 x]\big]ds\Big| \\
    &=\sup_{t\in I}\Big|\int^T_0 G_1(t,s)\big[[f(s,u)+g(s,v)+\lambda_1 u-\lambda_2 v]\\
    &\quad - [f(s,x)+g(s,y)+\lambda_1 x-\lambda_2 y]\big]\\
    &\quad - G_2(t,s)\big[[f(s,y)+g(s,x)+\lambda_1 y-\lambda_2 x]\\
    &\quad - [f(s,v)+g(s,u)+\lambda_1 v-\lambda_2 u]\big]ds\Big|\\
    &\leq\sup_{t\in I}\Big|\int^T_0 G_1(t,s)[\mu_1(u-x)+\mu_2(y-v)]\\
    &\quad - G_2(t,s)[\mu_1(y-v)+\mu_2(u-x)]ds\Big|\\
    &\leq 2(\mu_1+\mu_2)\cdot d_2\big((u,v), (x,y)\big)\\
    &\quad\times \sup_{t\in I}\Big|\int^T_0[G_1(t,s)-G_2(t,s)]ds\Big|\\
    &=2(\mu_1+\mu_2)\cdot d_2\big((u,v), (x,y)\big)\\
    &\quad\times\sup_{t\in I}\Big|\int^t_0\frac{e^{\tau_1(t-s)}}{1-e^{\tau_1T}}\, ds+
    \int^T_t\frac{e^{\tau_1(t+T-s)}}{1-e^{\tau_1T}}\Big|
\end{align*}
which yields, after integrating,
\begin{equation}\label{3.13}
    =\frac{2(\mu_1+\mu_2)}{\lambda_1+\lambda_2}\,d_2\big((u,v), (x,y)\big).
\end{equation}
Similarly, one obtains
\begin{equation}\label{3.14}
    d(A[y,x], A[v,u])\leq\frac{2(\mu_1+\mu_2)}{\lambda_1+\lambda_2}\,
    d_2\big((u,v), (x,y)\big).
\end{equation}
By summing up (\ref{3.13}) and (\ref{3.14}) we obtain that for all $x\geq u,\,y\leq v$:
$$  d_2(A[u,v], A[x,y])\leq \frac{\mu_1+\mu_2}{\lambda_1+\lambda_2}\,
    d_2\big((u,v), (x,y)\big),
$$
which proves that $A$ verifies the contraction condition \eqref{Bhas1} in Theorem \ref{th3}.

Now, let $(\alpha, \beta)\in X^2$ be a coupled lower-upper solution of (\ref{3.1}) and (\ref{3.2}). By Lemma \ref{lem2}, we have
$$  \alpha(t)\leq A[\alpha(t), \beta(t)]   $$
and
$$  \beta(t)\geq A[\beta(t), \alpha(t)],  $$
which show that all hypotheses of Theorem \ref{th3} and Theorem \ref{th5} are satisfied.\\
This shows that $A$ has a coupled fixed point in $X^2$, which, by Theorem \ref{th5}, is unique.
\end{proof}
\begin{remark}
\em Since the hypothesis of Theorem \ref{th6} are also satisfied, we deduce that the components of the fixed point $(u,v)$ are actually equal, that is, $u(t)\equiv v(t)$, which implies that $u=A[u,u]$ and hence $u(t)$ is the unique solution of
$$  u'(t)=f(t, u(t))+g(t, u(t))=h(t, u(t)),\quad t\in I.
$$
This establishes that the periodic BVP (\ref{3.1}) and (\ref{3.2}) has a unique solution on $I$.
\end{remark}
\begin{remark}\em
Note that our Theorem 3.1 is more general than Theorem 3.7 in \cite{Bha} since, if $\mu_1\neq \mu_2$, then
$$  \frac{\mu_1+\mu_2}{\lambda_1+\lambda_2} <
    \frac{2\max \{\mu_1, \mu_2\}}{\lambda_1+\lambda_2}\,.
$$
For example, if in Assumption 3.1 we have $\lambda_1=2$, $\lambda_2=3$, $\mu_1=1$ and $\mu_2=3$, then \eqref{3.4-1} holds:
$$  \frac{\mu_1+\mu_2}{\lambda_1+\lambda_2}=\frac{4}{5}<1,  $$
so Theorem 3.1 can be applied but, since
$$  k=\frac{2\max\{\mu_1, \mu_2\}}{\lambda_1+\lambda_2}=
    \frac{6}{5}>1,
$$
condition \eqref{Bhas} does not hold and hence Theorem 3.7 in \cite{Bha} cannot be applied.
\end{remark}
\begin{remark}
\em Note also that our contractive condition \eqref{Bhas1} is symmetric, while the contractive condition used in \cite{Bha} is not. Our generalization is based in fact on the idea of making the last one symmetric, which is very natural, as the great majority of contractive conditions in metrical fixed point theory are symmetric, see \cite{Rus2}.
\end{remark}

\vskip 0.5 cm {\it 

Department of Mathematics and Computer Science

North University of Baia Mare

Victoriei 76, 430122 Baia Mare ROMANIA

E-mail: vberinde@ubm.ro}
\end{document}